\documentclass[a4paper,14pt]{amsart}
\usepackage[frenchb]{babel}
\usepackage[utf8]{inputenc}
\usepackage[T1]{fontenc}
\usepackage{pdfsync}

\usepackage[colorlinks=true]{hyperref}
\hypersetup{urlcolor=green,linkcolor=brown,citecolor=blue,colorlinks=true}

\usepackage{pgf,tikz}
\usepackage{mathrsfs}
\usetikzlibrary{arrows}
\usepackage{amssymb,fancyhdr,txfonts,pxfonts}

\newcommand{\cqfd}{\hfill $\Box$}
\newtheorem{defi}{Définition}[section]
      
      \newtheorem{prop}[defi]{Proposition}
      \newtheorem{prop-def}[defi]{Proposition-Définition}
      \newtheorem{def-prop}[defi]{Définition-Proposition}
      \newtheorem{thm}[defi]{Théorème}
      
      \newtheorem{lem}[defi]{Lemme}
      \newtheorem{cor}[defi]{Corollaire}

\def\geod{{\text {géod}}}
\def\R{{\mathbf R}}
\def\Z{{\mathbf Z}}
\def\N{{\mathbf N}}

\title{Actions affines isométriques propres des groupes hyperboliques sur des espaces $\ell^{p}$}
\author{Aurélien Alvarez et Vincent Lafforgue}
\address{Aurélien Alvarez fait partie du projet ANR-14-CE25-0004 GAMME.\newline
MAPMO, UMR 7349, Université d'Orléans \newline Rue de Chartres, BP 6759 - 45067 Orléans cedex 2, France}
\address{Vincent Lafforgue fait partie du projet  ANR-14-CE25-0012 SINGSTAR.\newline
CNRS et Institut Fourier, UMR 5582, Université Grenoble Alpes \newline 100 rue des Maths, 38610 Gières, France}

\begin{document}

\maketitle

\begin{abstract}
Nous donnons une démonstration élémentaire et relativement courte du fait suivant : \textbf{tout groupe hyperbolique admet une action affine isométrique propre sur un espace $\ell^{p}$ pour $p$ suffisamment grand}.
Une première preuve de ce résultat a été donnée par Yu~\cite{yu}.
\end{abstract}


Les théorèmes de point fixe abondent dans la littérature mathématique et l'on sait combien ce sont des outils précieux ; citons par exemple le théorème de point fixe de Lefschetz en topologie algébrique qui permet de compter les points fixes d'une application continue d'un espace compact {\it via} l'action induite en homologie, le théorème du point fixe de Brouwer qui assure l'existence d'un point fixe pour toute application continue de la boule fermée, ou encore le non-moins célèbre théorème du point fixe de Banach aux innombrables applications, en particulier dans l'étude des équations différentielles.
Dans le contexte des actions de groupes, l'attention s'est d'abord portée sur les points fixes dans les espaces de Hilbert (avec la fameuse propriété (T) de Kazhdan \cite{kazhdan}) puis plus généralement dans les espaces de Banach.
Par exemple, dans \cite{bader-furman-gelander-monod}, Bader, Furman, Gelander et Monod démontrent que toute action affine isométrique d'un réseau de $\text{SL}(3,\R)$ sur~$\text{L}^p$ a un point fixe pour tout $1<p<\infty$, et un énoncé analogue a été démontré par V.~Lafforgue pour les actions affines à petite croissance exponentielle des sous-groupes co-compacts de $\text{SL}(3,\mathbf{Q}_p)$ sur des espaces de Banach de type > 1 \cite{renforcement-prop-T}, \cite{propTbanachique}. Voir également \cite{benben}, \cite{laat-salle1}, \cite{laat-salle2} pour d'autres généralisations.
Pour un panorama des développements récents concernant les propriétés de point fixe dans le cas des actions de groupe sur des espaces de Banach, on pourra se référer à l'excellent article \cite{nowak14} de Nowak.

De manière opposée en un certain sens, des groupes comme $\text{SO}(n,1)$ ou $\text{SU}(n,1)$, mais également leurs sous-groupes fermés, admettent des actions affines isométriques {\it propres} sur des espaces de Hilbert : on dit qu'ils ont la propriété de Haagerup.
Les groupes ayant cette propriété de Haagerup vérifient la conjecture de Baum-Connes à coefficients d'après Higson-Kasparov (voir par exemple \cite{valette} pour une introduction à la conjecture de Baum-Connes).

Les groupes hyperboliques se situent entre les deux catégories précédentes car certains ont la propriété de Haagerup mais d'autres la propriété (T).
Cependant tous admettent des actions affines isométriques propres sur des espaces $\ell^{p}$ ($1<p<\infty$).
Yu a d'abord démontré dans \cite{yu} que tout groupe hyperbolique admet une action affine propre sur un espace $\ell^{p}$ pour $p$ suffisamment grand, en s'appuyant sur des idées de Mineyev \cite{mineyev}.
Dans \cite{cornulier-tessera-valette}, Cornulier, Tessera et Valette démontrent que tout groupe de Lie $G$ simple de rang 1 (et de centre fini) possède une action affine propre sur $\text{L}^p(G)$ dont la partie linéaire est la représentation régulière, quel que soit le réel~$p$ strictement supérieur à la dimension conforme du bord de l'espace riemannien symétrique associé.
Nowak quant à lui démontre qu'un groupe a la propriété de Haagerup si et seulement s'il admet une action affine propre sur un espace de Banach $\text{L}^p([0,1])$ pour $1<p<2$ \cite{nowak09}.

Concernant les groupes hyperboliques, Bourdon prouve dans \cite{bourdon} qu'un tel groupe $G$ possède une action affine propre sur $\ell^p$ de la réunion disjointe d'un nombre fini de copies de $G$, pour tout $p$ strictement plus grand que la dimension conforme du bord de $G$.
Bourdon donne de plus dans \cite{bourdon-2} une minoration de la dimension conforme du bord à l'infini de certains groupes hyperboliques.
En utilisant les résultats de \cite{bourdon-pajot}, Nica démontre quant à lui qu'un groupe hyperbolique~$G$ admet une action affine propre sur $\text{L}^p(\partial G \times \partial G)$, où $\partial G$ désigne le bord de~G, pour $p$ assez grand.
En outre $G$ admet une action affine propre sur la 1-cohomologie-$\text{L}^p$ de~$G$ \cite{nica}.

\vspace{0.2cm}

Dans cet article, nous donnons une nouvelle démonstration du théorème de Yu \cite{yu}, comme conséquence du théorème~\ref{enonce-bon}.
La stratégie générale de la preuve est assez proche de celle de Yu mais notre démonstration est élémentaire et auto-contenue, alors que la preuve de Yu repose sur des travaux de Mineyev \cite{mineyev}.
Plus généralement on démontre le théorème~\ref{enonce-ppal} dont on ramène la démonstration à celle du théorème~\ref{enonce-bon}.

\tableofcontents

Afin de rendre cet article le plus largement accessible, nous avons rappelé certaines définitions à propos des groupes hyperboliques et donné quelques exemples importants.
Nous remercions Alain Valette et Yves Stalder pour leurs conseils et de nombreuses références.


\section{Actions affines isométriques }

Soit $\mathcal{E}$ un espace de Banach affine réel, c'est-à-dire la donnée d'un espace de Banach réel $\mathcal{E}^{\circ}$ opérant simplement transitivement dans un ensemble $\mathcal{E}$.
La donnée d'une origine $\xi_{\circ}$ dans $\mathcal{E}$ permet d'identifier $\mathcal{E}$ à $\mathcal{E}^{\circ}$, ainsi que le groupe $\text{Isom}(\mathcal{E})$ des isométries affines de~$\mathcal{E}$ à $\text{Isom}(\mathcal{E}^{\circ}) \ltimes \mathcal{E}^{\circ}$, où $\text{Isom}(\mathcal{E}^{\circ})$ désigne le groupe des isométries vectorielles de $\mathcal{E}^{\circ}$.

Une action {\it affine} (isométrique) d'un groupe topologique $G$ sur $\mathcal{E}$ est un homomorphisme de groupes $\alpha : G \longrightarrow \text{Isom}(\mathcal{E})$ tel que, pour tout $\xi$ de $\mathcal{E}$,
$$g \longmapsto \alpha(g)\cdot \xi$$
\noindent
est continue.
L'identification précédente permet d'écrire $\alpha(g)(\cdot) = \pi(g)(\cdot) + c(g)$ pour tout~$g$ de~$G$, où $\pi$ est une représentation isométrique de $G$ dans l'espace de Banach~$\mathcal{E}^{\circ}$ et $c : G \longrightarrow \mathcal{E}^{\circ}$ un {\it cocycle} continu dans le sens suivant : pour tous $g_1, g_2$ de $G$, on a
$$c(g_1 g_2) = c(g_1) + \pi(g_1) \cdot c(g_2).$$

Pour tout point $\xi$ de l'espace affine $\mathcal{E}$, il existe un unique vecteur $\xi^{\circ}$ de $\mathcal{E}^{\circ}$ tel que $\xi = \xi_{\circ} + \xi^{\circ}$.
Si $\eta^{\circ}$ désigne un vecteur quelconque de $\mathcal{E}^{\circ}$, on a alors
$$\begin{array}{rcl}
\alpha(g) \cdot \xi & = & \pi(g) \cdot \xi^{\circ} + c(g) + \xi_{\circ}\\
& = & \pi(g) \cdot ({\xi}^{\circ}+\eta^{\circ}) - \pi(g) \cdot \eta^{\circ} + c(g) + \eta^{\circ} + (\xi_{\circ} - \eta^{\circ}),\\
\end{array}$$

\noindent
autrement dit $\alpha(g)(\cdot) = \pi(g)(\cdot) + c(g) + b(g)$, où $b(g) = \eta^{\circ} - \pi(g) \cdot \eta^{\circ}$ est un {\it cobord}, c'est-à-dire un cocycle de la forme $h \longmapsto \eta^{\circ} - \pi(h) \cdot \eta^{\circ}$ pour un certain vecteur~$\eta^{\circ}$ de $\mathcal{E}^{\circ}$.
Ainsi, changer d'origine dans $\mathcal{E}$ revient à modifier le cocycle de départ en lui additionnant un cobord.
On en déduit donc que l'espace vectoriel quotient $\text{H}^1(G,\pi)$ des cocycles modulo les cobords classifie les actions affines de $G$ de partie linéaire~$\pi$, à isomorphisme près.
En particulier, une action affine $\alpha = (\pi,c)$ a un point fixe dans $\mathcal{E}$ si et seulement si le cocycle $c$ est un cobord.
Remarquons également que l'inégalité triangulaire entraîne immédiatement que tout cobord est borné (par deux fois la norme du vecteur $\eta^{\circ}$ avec les notations précédentes).

L'action affine $\alpha = (\pi,c)$ de $G$ sur $\mathcal{E}$ est dite {\it propre} si le cocycle $c$ est propre, c'est-à-dire si $\lim_{g \to \infty} \| c(g) \| = \infty$.
C'est la situation qui va nous intéresser dans cet article.

\vspace{0.2cm}

\noindent
\textbf{Une méthode utile pour construire des actions affines.}
Soit $(\mathcal{E}^{\circ},\| \cdot \|)$ un espace de Banach et $\pi$ une action linéaire isométrique d'un groupe topologique $G$ sur $\mathcal{E}^{\circ}$.
Supposons donné un espace vectoriel $V$, un plongement d'espace vectoriel de $\mathcal{E}^{\circ}$ dans~$V$ et une action linéaire de $G$ sur $V$ qui stabilise $\mathcal{E}^{\circ}$ et dont la restriction à $\mathcal{E}^{\circ}$ est~$\pi$.
Si $\xi_{\circ}$ est un élément de $V$ tel que, pour tout $g$ de $G$, le vecteur $\xi_{\circ} - g \cdot \xi_{\circ}$ appartient au sous-espace $\mathcal{E}^{\circ}$ et $g \longmapsto \xi_{\circ} - g \cdot \xi_{\circ}$ est continue de $G$ dans $\mathcal{E}^{\circ}$, alors l'action linéaire de $G$ sur $V$ se restreint en une action affine de $G$ sur l'espace de Banach affine $\mathcal{E} = \xi_{\circ} + \mathcal{E}^{\circ}$.
Cette dernière est aussi donnée par le cocycle $g \longmapsto \xi_{\circ} - g \cdot \xi_{\circ}$.
Par définition, l'action est propre si et seulement si $\lim_{g \to \infty} \| \xi_{\circ} - g \cdot \xi_{\circ} \| = \infty$.

Remarquons enfin que l'inégalité triangulaire interdit à $\xi_{\circ}$ d'appartenir à $\mathcal{E}^{\circ}$ dès que la borne supérieure de l'ensemble des $\| \xi_{\circ} - g \cdot \xi_{\circ} \|$ pour $g$ dans $G$ est infinie.

\vspace{0.2cm}

\noindent
\textbf{Exemple :} Supposons $G$ discret de type fini ; désignons par~$\mathcal{E}^{\circ}$ l'espace de Banach $\ell^\infty(G)$, par $V$ l'espace vectoriel des fonctions réelles sur $G$ et par~$\xi_{\circ}$ la fonction $d(1,\cdot)$, où $d$ désigne la métrique des mots associée à un système fini de générateurs de $G$.
La métrique $d$ étant invariante à gauche, pour tout $g$ de $G$, la fonction $c(g) = \xi_{\circ} - g \cdot \xi_{\circ}$ se réécrit
$$c(g) = d(1,\cdot) - d(1,g^{-1} \cdot) = d(1,\cdot) - d(g,\cdot).$$
De l'inégalité triangulaire, on déduit immédiatement que la fonction $c(g)$ est bornée sur~$G$ par $d(1,g)$ : la méthode précédente s'applique et donne une action affine de~$G$ sur l'espace de Banach affine $d(1,\cdot)+\ell^\infty(G)$.
Par ailleurs on calcule facilement que la norme infinie $\| c(g) \|_{\infty}$ de $c(g)$ est exactement égale à $d(1,g)$, ce qui permet de conclure que l'action est propre.


\section{Petit détour dans les arbres}

Soit $X$ un arbre\footnote{On désignera tout aussi bien par $X$ l'arbre que l'ensemble de ses sommets.} (le lecteur ne perdra rien à penser à l'arbre de Cayley d'un groupe libre sur deux générateurs). Étant donné deux sommets $x$ et $a$ de $X$, on considère la mesure de probabilité \textit{masse de Dirac} $\mu_x(a)$ portée par le sommet voisin\footnote{comprendre ici à distance 1 si $a$ et $x$ sont distincts et à distance 0 sinon} de~$a$ le plus proche de~$x$.


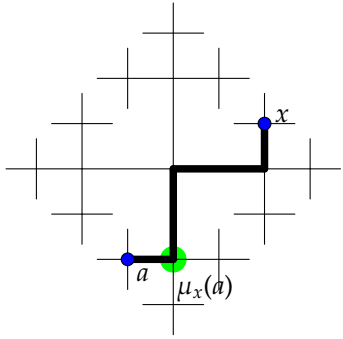
\begin{center}
\definecolor{qqffqq}{rgb}{0.,1.,0.}
\definecolor{qqqqff}{rgb}{0.,0.,1.}
\begin{tikzpicture}[line cap=round,line join=round,>=triangle 45,x=1.0cm,y=1.0cm,scale=0.4]
\clip(-6.,-6.) rectangle (6.,6.);
\draw [color=qqffqq,fill=qqffqq,fill opacity=1.0] (0.,-3.) circle (0.42cm);
\draw [line width=2.8pt] (0.,0.)-- (3.,0.);
\draw (0.,0.)-- (-3.,0.);
\draw (0.,0.)-- (0.,3.);
\draw [line width=2.8pt] (0.,0.)-- (0.,-3.);
\draw [line width=2.8pt] (3.,0.)-- (3.,1.5);
\draw (3.,0.)-- (3.,-1.5);
\draw (3.,1.5)-- (4.,1.5);
\draw (3.,1.5)-- (2.,1.5);
\draw (3.,1.5)-- (3.,2.5);
\draw (3.,-1.5)-- (4.,-1.5);
\draw (3.,-1.5)-- (3.,-2.5);
\draw (3.,-1.5)-- (2.,-1.5);
\draw (3.,0.)-- (4.5,0.);
\draw (4.5,0.)-- (5.5,0.);
\draw (4.5,0.)-- (4.5,-1.);
\draw (4.5,0.)-- (4.5,1.);
\draw (-3.,0.)-- (-3.,-1.5);
\draw (-3.,-1.5)-- (-4.,-1.5);
\draw (-3.,-1.5)-- (-3.,-2.5);
\draw (-3.,-1.5)-- (-2.,-1.5);
\draw (-3.,1.5)-- (-2.,1.5);
\draw (-3.,0.)-- (-3.,1.5);
\draw (-3.,1.5)-- (-4.,1.5);
\draw (-3.,1.5)-- (-3.,2.5);
\draw (-3.,0.)-- (-4.5,0.);
\draw (-4.5,0.)-- (-4.5,1.);
\draw (-4.5,0.)-- (-4.5,-1.);
\draw (-4.5,0.)-- (-5.5,0.);
\draw (0.,3.)-- (0.,4.5);
\draw (0.,4.5)-- (1.,4.5);
\draw (0.,4.5)-- (-1.,4.5);
\draw (0.,4.5)-- (0.,5.5);
\draw (-1.5,3.)-- (-1.5,4.);
\draw (-1.5,3.)-- (-1.5,2.);
\draw (-1.5,3.)-- (-2.5,3.);
\draw (0.,3.)-- (-1.5,3.);
\draw (1.5,3.)-- (1.5,2.);
\draw (1.5,3.)-- (1.5,4.);
\draw (1.5,3.)-- (2.5,3.);
\draw (0.,3.)-- (1.5,3.);
\draw (0.,-3.)-- (0.,-4.5);
\draw (0.,-4.5)-- (-1.,-4.5);
\draw (0.,-4.5)-- (0.,-5.5);
\draw (0.,-4.5)-- (1.,-4.5);
\draw (0.,-3.)-- (1.5,-3.);
\draw (1.5,-3.)-- (1.5,-4.);
\draw (1.5,-3.)-- (1.5,-2.);
\draw (1.5,-3.)-- (2.5,-3.);
\draw [line width=2.8pt] (0.,-3.)-- (-1.5,-3.);
\draw (-1.5,-3.)-- (-1.5,-2.);
\draw (-1.5,-3.)-- (-2.5,-3.);
\draw (-1.5,-3.)-- (-1.5,-4.);
\draw (3.05,2.3) node[anchor=north west] {$x$};
\draw (-1.5,-3) node[anchor=north west] {$a$};
\draw (-0.2,-3.2) node[anchor=north west] {$\mu_x(a)$};
\begin{scriptsize}
\draw [fill=qqqqff] (3.,1.5) circle (6pt);
\draw [fill=qqqqff] (-1.5,-3.) circle (6pt);
\end{scriptsize}
\end{tikzpicture}
\end{center}

L'application $\mu$ de $X \times X$ dans l'espace $\mathcal{P}(X)$ des mesures de probabilité sur $X$ ainsi définie est telle que, pour tout $a$ et tous $x, x'$ distincts de $X$ :

\vspace{0.1cm}

\begin{enumerate}
\renewcommand{\theenumi}{\roman{enumi}}
\item le support de $\mu_x(a)$ est contenu dans la boule fermée centrée en $a$ et de rayon 1 (et même sur la sphère centrée en $a$ et de rayon 1 dès que $a$ et $x$ sont distincts) ;
\item $\| \mu_x(a) - \mu_{x'}(a) \|_1 = 2$ si $a$ est sur la géodésique de $x$ à $x'$, et 0 sinon ;
\item les supports de $\mu_x(a)$ et $\mu_{x'}(a)$ sont disjoints si et seulement si $a$ est sur la géodésique de $x$ à $x'$.
\end{enumerate}

Les trois items sont immédiats à vérifier.

\vspace{0.1cm}

Soit $G$ un groupe opérant dans l'arbre $X$ si bien que l'application $\mu$ est de fait $G$-équivariante.
Soit $o$ un sommet de $X$.
Notons $X^{\leq 1}$ l'ensemble des couples de sommets de $X$ à distance 0 ou 1 l'un de l'autre\footnote{L'ensemble $X^{\leq 1}$ s'identifie naturellement à la réunion disjointe des sommets et des arêtes (orientées) de $X$.}, $\mathcal{E}^{\circ}$ l'espace de Banach $\ell^p(X^{\leq 1})$ ($p \in [1,\infty[$), $V$ l'espace vectoriel des fonctions sur $X^{\leq 1}$ et $\xi_{\circ} : (x,y) \longmapsto \mu_o(x)(\{y\})$ une fonction sur $X \times X$ que l'on ne perd rien à considérer comme un élément de~$V$ d'après l'item (i) ; notons enfin, pour tout $g$ de $G$, $c(g) = \xi_{\circ} - g \cdot \xi_{\circ}$.
En utilisant la $G$-équivariance de $\mu$, la fonction $c(g)$ se réécrit $c(g) = \mu_o - \mu_{g \cdot o}$.
L'item (ii) ci-dessus assure que la fonction $c(g)$ est nulle en dehors du segment géodésique de~$o$ à $g \cdot o$.
Bien entendu, si $o$ est un point fixe de $g$, la fonction $c(g)$ est nulle.
Dans le cas contraire, on en déduit facilement que $\| c(g) \|_p = 2 (d(o,g \cdot o)+1)^{1/p}$.
Si de plus l'action de $G$ sur $X$ est propre ({\it i.e.} la distance $d(o,g \cdot o)$ dans l'arbre tend vers l'infini quand $g$ tend vers l'infini), on en déduit alors une action affine propre de~$G$ sur l'espace de Banach affine $\mu_o+\ell^p(X^{\leq 1})$.

\vspace{0.2cm}

Nous allons voir l'analogue de la construction précédente de l'application $\mu$ dans le monde \og continu \fg\ et conforter ainsi notre intuition géométrique de la situation dans le cas des espaces euclidiens et hyperboliques.

\vspace{0.1cm}


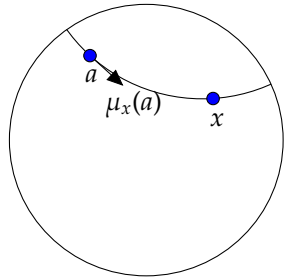
\begin{center}
\definecolor{qqqqff}{rgb}{0.,0.,1.}
\begin{tikzpicture}[line cap=round,line join=round,>=triangle 45,x=1.0cm,y=1.0cm,scale=1.8]
\clip(-1.2,-1.2) rectangle (1.2,1.2);
\draw(0.,0.) circle (1.cm);
\draw [shift={(0.40986320234477586,1.5211947052124717)}] plot[domain=3.7615291315932247:5.136878399080434,variable=\t]({1.*1.217382920778328*cos(\t r)+0.*1.217382920778328*sin(\t r)},{0.*1.217382920778328*cos(\t r)+1.*1.217382920778328*sin(\t r)});
\draw [->] (-0.41112834707886825,0.6223098118152487) -- (-0.16714946259663577,0.3994730460968994);
\draw (-0.37,0.43) node[anchor=north west] {$\mu_x(a)$};
\draw (-0.51,0.59) node[anchor=north west] {$a$};
\draw (0.4,0.27) node[anchor=north west] {$x$};
\begin{scriptsize}
\draw [fill=qqqqff] (-0.41112834707886825,0.6223098118152487) circle (1.3pt);
\draw [fill=qqqqff] (0.48818738430852365,0.30633401430075957) circle (1.3pt);
\end{scriptsize}
\end{tikzpicture}
\end{center}

\noindent
Désignons d'abord par $X$ une variété riemannienne simplement connexe à courbure négative ou nulle et par $\Gamma(X)$ l'espace vectoriel des champs de vecteurs sur~$X$.
Notons $\mathcal{E}^{\circ}$ l'espace de Banach des classes (modulo égalité presque partout) de champs de vecteurs $p$-intégrables ($p \in [1,\infty[$) et $V$ l'espace vectoriel des classes de champs de vecteurs.
On définit une application $\mu : X \longrightarrow \Gamma(X)$ en considérant, pour tout~$x$ de $X$ le champ géodésique $\mu_x = -\text{grad} \ d(x,\cdot)$ qui pointe dans la direction de~$x$ (par construction, $a \longmapsto \mu_x(a) \in \text{T}_a X$ est un champ de vecteurs singulier en~$x$).

Fixons $o$ un point de $X$ ; notons alors $\xi_{\circ} = \mu_o$ et, pour tout $g$ dans le groupe d'isométries $G$ de $X$, $c(g) = \xi_{\circ} - g \cdot \xi_{\circ}$.
L'application $(x,a) \longmapsto \mu_x(a)$ est $G$-équivariante par construction, donc $c(g) = \mu_o - \mu_{g \cdot o}$.
L'action de $G$ sur $X$ étant propre, pour pouvoir conclure comme expliqué dans la méthode du paragraphe~1, il reste donc à voir que, pour tous $x,x'$ de $X$, le champ de vecteurs $\mu_x - \mu_{x'}$ est $\text{L}^p$-intégrable pour~$p$ suffisamment grand et de norme $\text{L}^p$ tendant vers l'infini quand la distance de $x$ à $x'$ tend vers l'infini.

\noindent
\textbf{Cas euclidien :}
Prenons comme origine de l'espace euclidien $\mathbf{R}^n$ le milieu du segment $[x,x']$.
Pour un point $a$ de $\mathbf{R}^n$ à une distance $r$ donnée de l'origine, c'est sur l'hyperplan médiateur de $[x,x']$ que l'angle $\measuredangle xax'$ est maximal, de mesure $2 \arctan(d(x,x')/{2r}) \sim d(x,x')/r$ quand $r$ tend vers l'infini.
L'élément de volume étant proportionnel à $r^{n-1} dr d\theta$, où $\theta$ désigne une coordonnée angulaire ($\theta \in \mathbf{S}^{n-1}$), on en déduit l'intégrabilité dès que $(n-1)-p < -1$, {\it i.e.} $p>n$.

\noindent
\textbf{Cas hyperbolique :}
On utilise le même argument que précédemment.
Sur l'hyperplan médiateur de $[x,x']$ dans l'espace hyperbolique réel $\mathcal{H}^n_{\mathbf{R}}$, à une constante multiplicative près, l'angle $\measuredangle xax'$ est équivalent à $\sinh(r)^{-1}$.
L'élément de volume étant proportionnel à $\sinh(r)^{n-1} dr d\theta$, où $\theta$ désigne une coordonnée angulaire dans l'espace tangent au point en question, on en déduit l'intégrabilité dès que $(n-1)-p < 0$, {\it i.e.} $p>n-1$.
En particulier, on remarque que la condition est satisfaite pour tout $p > 1$ dans le cas du plan hyperbolique.


\section{Espaces hyperboliques : définitions et lemmes}\label{definitions-lemmes}

Nous rappelons quelques définitions à propos des espaces hyperboliques et démontrons des lemmes qui nous seront utiles par la suite.
Deux références classiques sur le sujet sont \cite{delzanthyper} et \cite{harpehyper}.

\begin{defi}
Soit $\delta$ un réel positif\footnote{Dans tout ce texte, {\it positif} signifie {\it positif ou nul}, de même que {\it supérieur} signifie {\it supérieur ou égal}. Mêmes conventions pour {\it inférieur} et {\it négatif}.}.
Un espace métrique $(X,d)$ est dit $\delta$-hyperbolique si, pour tout quadruplet $(x_1,x_2,x_3,x_4)$ de points de $X$, on a
$$d(x_1,x_4)+d(x_2,x_3) \leq \max \big( d(x_1,x_2)+d(x_3,x_4),d(x_1,x_3)+d(x_2,x_4) \big) + \delta.$$
\end{defi}

Nous dirons qu'un espace métrique $(X,d)$ est {\it hyperbolique} s'il existe un réel~$\delta$ positif tel que $(X,d)$ soit $\delta$-hyperbolique.


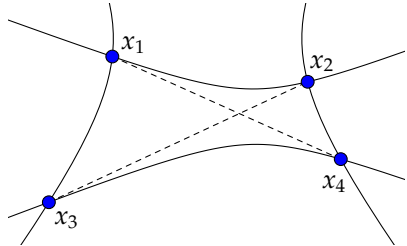
\begin{center}
\definecolor{qqqqff}{rgb}{0.,0.,1.}
\begin{tikzpicture}[line cap=round,line join=round,>=triangle 45,x=1.0cm,y=1.0cm,scale=0.8]
\clip(4.5,1.5) rectangle (11.,5.5);
\draw [samples=50,domain=-0.99:0.99,rotate around={90.4275725068334:(8.540896221273629,3.6190270303925853)},xshift=8.540896221273629cm,yshift=3.6190270303925853cm] plot ({0.4592973647674953*(1+(\x)^2)/(1-(\x)^2)},{1.1844767895219772*2*(\x)/(1-(\x)^2)});
\draw [samples=50,domain=-0.99:0.99,rotate around={90.4275725068334:(8.540896221273629,3.6190270303925853)},xshift=8.540896221273629cm,yshift=3.6190270303925853cm] plot ({0.4592973647674953*(-1-(\x)^2)/(1-(\x)^2)},{1.1844767895219772*(-2)*(\x)/(1-(\x)^2)});
\draw [samples=50,domain=-0.99:0.99,rotate around={0.:(7.768641750508568,4.914512306446116)},xshift=7.768641750508568cm,yshift=4.914512306446116cm] plot ({1.5490446355363514*(1+(\x)^2)/(1-(\x)^2)},{2.001942508980508*2*(\x)/(1-(\x)^2)});
\draw [samples=50,domain=-0.99:0.99,rotate around={0.:(7.768641750508568,4.914512306446116)},xshift=7.768641750508568cm,yshift=4.914512306446116cm] plot ({1.5490446355363514*(-1-(\x)^2)/(1-(\x)^2)},{2.001942508980508*(-2)*(\x)/(1-(\x)^2)});
\draw [dash pattern=on 2pt off 2pt] (5.159088089519433,2.2004592116979262)-- (9.414129175066414,4.197170444915182);
\draw [dash pattern=on 2pt off 2pt] (6.202411125823259,4.61547653075052)-- (9.958003017110203,2.9149806365502178);
\draw (6.2,5.1) node[anchor=north west] {$x_1$};
\draw (9.3,4.8) node[anchor=north west] {$x_2$};
\draw (5.1,2.2) node[anchor=north west] {$x_3$};
\draw (9.5,2.8) node[anchor=north west] {$x_4$};
\begin{scriptsize}
\draw [fill=qqqqff] (9.414129175066414,4.197170444915182) circle (3pt);
\draw [fill=qqqqff] (5.159088089519433,2.2004592116979262) circle (3pt);
\draw [fill=qqqqff] (6.202411125823259,4.61547653075052) circle (3pt);
\draw [fill=qqqqff] (9.958003017110203,2.9149806365502178) circle (3pt);
\end{scriptsize}
\end{tikzpicture}
\end{center}

Géométriquement, la définition d'hyperbolicité signifie que la somme des longueurs des diagonales d'un quadrilatère n'excède jamais le maximum de la somme des longueurs des côtés opposés, à une constante $\delta$ donnée près.
Les arbres et les espaces hyperboliques $\mathcal{H}^n$ sont les prototypes mêmes d'espaces hyperboliques au sens de la définition ci-dessus.
Le théorème de Pythagore assure que la diagonale d'un carré de côté $\ell$ est de longueur $\sqrt 2 \cdot \ell$ : un plan euclidien n'est donc pas hyperbolique.

\vspace{0.2cm}

Une fonction $f$ d'un espace métrique $X$ dans un espace métrique $X'$ est une {\it quasi-isométrie} si
\begin{enumerate}
\item la distance entre les images de deux points de $X$ est encadrée par les fonctions affines $t \longmapsto \alpha^{-1} \cdot t - \beta$ et $t \longmapsto \alpha \cdot t + \beta$ ($\alpha \geq 1$ et $\beta \geq 0$) de la distance dans~$X$ entre les deux points ;
\item tout point de $X'$ est à une distance de $f(X)$ inférieure à une constante donnée.
\end{enumerate}

Deux espaces métriques $X$ et $X'$ sont {\it quasi-isométriques} s'il existe une quasi-isométrie de $X$ dans $X'$.
C'est un fait fondamental (mais non trivial) : l'hyperbolicité est un invariant de quasi-isométrie pour les espaces (faiblement) géodésiques (voir par exemple \cite{harpehyper}, thm.~29, chap.~1), ce qui assure le bien-fondé de la définition suivante.

\begin{defi}
Un groupe $\Gamma$ de type fini est hyperbolique si, pour un (ou pour tout) système fini de générateurs, son graphe de Cayley est un espace hyperbolique pour la distance des mots.
\end{defi}

C'est à Gromov, à la suite de Dehn, Mostow, Thurston, Cannon, Rips que l'on doit l'idée d'étudier les groupes hyperboliques dans le cadre d'une définition aussi générale que celle-ci \cite{gromov}. Parmi les exemples fondamentaux de groupes hyperboliques, mentionnons les groupes libres de rang fini et les groupes de surface de genre au moins 2.

\vspace{0.2cm}

Rappelons quelques notations et définitions classiques dans les espaces métriques.
Soit $(X,d)$ un espace métrique.

\begin{itemize}
\renewcommand{\labelitemi}{$\bullet$}

\item Si $x$ est un point de $X$, $r$ un réel positif, on note respectivement $B(x,r)=\{y\in X ; d(x,y)\leq r\}$ et $S(x,r)=\{y\in X ; d(x,y)= r\}$ la boule fermée et la sphère de centre $x$ et de rayon~$r$.
L'espace $(X,d)$ est dit {\it uniformément localement fini} si, pour tout réel positif $r$, il existe un entier $K$ tel que, pour tout $x$ de $X$, la boule $B(x,r)$ contient au plus $K$ points.

\item Pour tout entier naturel $R$, on note $X^{\leq R}$ l'ensemble $\{(x,y)\in X \times X ; d(x,y)\leq R\}$ des couples de points à distance au plus $R$ l'un de l'autre.

\item Si $\epsilon$ est un réel positif et $x,y$ sont deux points de $X$, on désigne par $\epsilon\text{-}\geod(x,y)$ l'ensemble des points $z$ de $X$ tels que
$$d(x,z)+d(z,y) \leq d(x,y)+\epsilon.$$
Pour simplifier la notation, on note $\geod(x,y)$ au lieu de $0\text{-}\geod(x,y)$ l'ensemble des points $z$ de $X$ tels que $d(x,z)+d(z,y)=d(x,y)$.

\item Soit $\delta$ un réel positif.
L'espace $(X,d)$ est dit {\it faiblement $\delta$-géodésique} si, pour tous points $x,y$ de $X$ et pour
tout réel $s$ de $[0,d(x,y)+\delta]$, il existe un élément~$z$ de~$X$ tel que $d(x,z)\leq s$ et $d(z,y)\leq d(x,y)-s+\delta$.
L'espace $(X,d)$ est dit {\it faiblement géodésique} dès lors qu'il existe un réel positif $\delta$ tel que $(X,d)$ est faiblement $\delta$-géodésique.

\end{itemize}

\vspace{0.2cm}

Nous énonçons maintenant trois lemmes que nous utiliserons dans la démonstration du théorème principal.
Pour cela, désignons par $(X,d)$ un espace $\delta$-hyperbolique. 

\begin{lem}\label{sphere-alpha-geodesique}
Soit $\alpha$ un réel positif et $a,x,y,y'$ quatre points de $X$. Si $y$ et $y'$ appartiennent à la même sphère de centre $a$ ($d(a,y)=d(a,y')$) et à $\alpha\text{-}\geod(a,x)$, alors $d(y,y')\leq \alpha+\delta$. 
\end{lem}


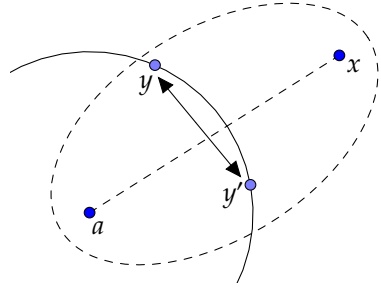
\begin{center}
\definecolor{xdxdff}{rgb}{0.49019607843137253,0.49019607843137253,1.}
\definecolor{qqqqff}{rgb}{0.,0.,1.}
\begin{tikzpicture}[line cap=round,line join=round,>=triangle 45,x=1.0cm,y=1.0cm,scale=0.5]
\clip(5.5,4.) rectangle (16.,12.);
\draw [shift={(7.57,5.82)}] plot[domain=-1.0092391694487066:2.132353484141087,variable=\t]({1.*4.30043021103703*cos(\t r)+0.*4.30043021103703*sin(\t r)},{0.*4.30043021103703*cos(\t r)+1.*4.30043021103703*sin(\t r)});
\draw [rotate around={32.4034267641815:(10.855,7.935)},dash pattern=on 3pt off 3pt] (10.855,7.935) ellipse (4.777194093289525cm and 2.7718465695201133cm);
\draw [dash pattern=on 3pt off 3pt] (7.57,5.85)-- (14.14,10.02);
\draw [->] (10.360891668884749,8.240731966994943) -- (9.36,9.46);
\draw [->] (10.360891668884749,8.240731966994943) -- (11.56,6.78);
\draw (14.1,10.17) node[anchor=north west] {$x$};
\draw (7.4,5.88) node[anchor=north west] {$a$};
\draw (8.6,9.75) node[anchor=north west] {$y$};
\draw (10.8,6.9) node[anchor=north west] {$y'$};
\begin{scriptsize}
\draw [fill=qqqqff] (7.57,5.85) circle (4pt);
\draw [fill=qqqqff] (14.14,10.02) circle (4pt);
\draw [fill=xdxdff] (9.2894724653713,9.761714657459358) circle (4pt);
\draw [fill=xdxdff] (11.801386672598293,6.587506890495156) circle (4pt);
\end{scriptsize}
\end{tikzpicture}
\end{center}

\begin{proof}
C'est une conséquence immédiate des définitions :
$$d(a,x) + d(y,y') \leq \max \big( \underbrace{\underbrace{d(a,y)}_{=d(a,y')}+d(y',x)}_{\leq d(a,x)+\alpha},\underbrace{\underbrace{d(a,y')}_{=d(a,y)}+d(y,x)}_{\leq d(a,x)+\alpha} \big) + \delta.$$
\end{proof}

\noindent
Remarque : le lemme~\ref{sphere-alpha-geodesique} justifie d'une certaine manière de dessiner l'$\epsilon$-géodésique entre deux points par un ensemble \og patatoïdal \fg\ de largeur $\epsilon+\delta$ étendu entre les deux points.

\begin{lem}\label{chemin-alpha-geodesique}
Soit $\alpha,\beta$ deux réels positifs et $a,x,y,z$ quatre points de $X$. Si $y$ appartient à $\alpha\text{-}\geod(a,x)$, $z$ à $\beta\text{-}\geod(a,y)$ et si $d(y,z)\geq \frac{\alpha+\beta}{2}$, alors $z$ appartient à $(\beta+\delta)\text{-}\geod(a,x)$. 
\end{lem}


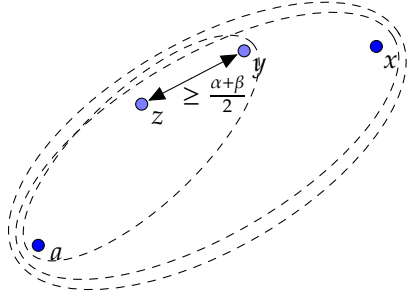
\begin{center}
\definecolor{xdxdff}{rgb}{0.49019607843137253,0.49019607843137253,1.}
\definecolor{qqqqff}{rgb}{0.,0.,1.}
\begin{tikzpicture}[line cap=round,line join=round,>=triangle 45,x=1.0cm,y=1.0cm,scale=0.45]
\clip(2.5,2.5) rectangle (15.5,12.);
\draw [rotate around={30.67291750972297:(9.2,7.09)},dash pattern=on 3pt off 3pt] (9.2,7.09) ellipse (6.506179515467078cm and 3.056447592791918cm);
\draw [rotate around={43.60272836202584:(7.268146282572566,7.024892149301136)},dash pattern=on 3pt off 3pt] (7.268146282572566,7.024892149301136) ellipse (4.4830914477168395cm and 1.6856920964723259cm);
\draw [rotate around={30.67291750972297:(9.2,7.09)},dash pattern=on 3pt off 3pt] (9.2,7.09) ellipse (6.316205776567683cm and 2.627918456106874cm);
\draw [->] (8.74160219742885,9.084924461311076) -- (10.086318194346292,9.790153295249999);
\draw [->] (8.74160219742885,9.084924461311076) -- (7.470991982183559,8.418559992871323);
\draw (8.2,9.430375966538508) node[anchor=north west] {$\geq \frac{\alpha+\beta}{2}$};
\draw (4.35,4.386100233231534) node[anchor=north west] {$a$};
\draw (14.078441884389509,10.051499471012873) node[anchor=north west] {$x$};
\draw (7.302549108305516,8.395170125747898) node[anchor=north west] {$z$};
\draw (10.219947386897234,9.938567924744806) node[anchor=north west] {$y$};
\begin{scriptsize}
\draw [fill=qqqqff] (4.26,4.16) circle (5pt);
\draw [fill=qqqqff] (14.14,10.02) circle (5pt);
\draw [fill=xdxdff] (10.276292565145148,9.889784298602288) circle (5pt);
\draw [fill=xdxdff] (7.281017611384707,8.318928989519035) circle (5pt);
\end{scriptsize}
\end{tikzpicture}
\end{center}

\begin{proof}
La propriété d'hyperbolicité pour $a,x,y,z$ donne
$$d(x,z)+d(y,a) \leq \max(d(x,y)+d(z,a), d(x,a)+d(y,z))+\delta,$$
\noindent
d'où, en ajoutant $d(z,a)$ de part et d'autre de l'inégalité et en utilisant le fait que $z \in \beta\text{-}\geod(y,a)$, on obtient
$$d(x,z)+d(y,a)+d(z,a) \leq \max(d(x,y)+ 2 \cdot d(z,a), d(x,a)+\underbrace{d(y,z)+d(z,a)}_{\leq d(a,y)+\beta})+\delta.$$
Puis, en retranchant $d(y,a)$ de part et d'autre de l'inégalité, et en utilisant le fait que $y \in \alpha\text{-}\geod(x,a)$ et $z \in \beta\text{-}\geod(y,a)$, on en déduit
$$d(x,z)+d(z,a) \leq \max(\underbrace{d(x,y)+d(y,a)}_{\leq d(a,x) + \alpha}+\underbrace{2 \cdot (d(z,a)-d(y,a))}_{\leq 2 \cdot (\beta-d(y,z))}, d(x,a) + \beta)+\delta,$$
\noindent
d'où finalement,
$$d(x,z)+d(z,a) \leq d(a,x) + \max(\alpha+2 \beta - 2 \cdot d(y,z), \beta)+\delta.$$
L'hypothèse $d(y,z)\geq \frac{\alpha+\beta}{2}$ permet de conclure.
\end{proof}

\begin{lem}\label{milieu-geodesique}
Soit $\epsilon$ un réel positif et $a,x,x'$ trois points de $X$. Si $a$ appartient à $\epsilon\text{-}\geod(x,x')$, alors
$$\epsilon\text{-}\geod(x,a)\cap\epsilon\text{-}\geod(x',a) \subset B(a,3\epsilon/2).$$
\end{lem}

\noindent
Remarquons que le lemme précédent n'utilise pas l'hyperbolicité de $X$.


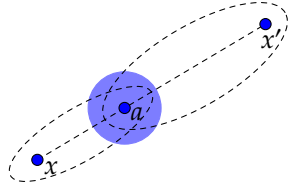
\begin{center}
\definecolor{xdxdff}{rgb}{0.49019607843137253,0.49019607843137253,1.}
\definecolor{qqqqff}{rgb}{0.,0.,1.}
\begin{tikzpicture}[line cap=round,line join=round,>=triangle 45,x=1.0cm,y=1.0cm,scale=0.6]
\clip(1.5,1.5) rectangle (9.5,7.5);
\draw [color=xdxdff,fill=xdxdff,fill opacity=1.0] (4.911764705882353,4.147058823529412) circle (0.8cm);
\draw [dash pattern=on 2pt off 2pt] (3.,3.)-- (8.,6.);
\draw [rotate around={30.963756532073504:(3.9558823529411775,3.5735294117647065)},dash pattern=on 2pt off 2pt] (3.9558823529411775,3.5735294117647065) ellipse (1.8cm and 0.5781024825263787cm);
\draw [rotate around={30.963756532073543:(6.455882352941177,5.073529411764706)},dash pattern=on 2pt off 2pt] (6.455882352941177,5.073529411764706) ellipse (2.3cm and 0.8784675498067411cm);
\draw (2.95,3.18) node[anchor=north west] {$x$};
\draw (7.7,6.1) node[anchor=north west] {$x'$};
\draw (4.85,4.36) node[anchor=north west] {$a$};
\begin{scriptsize}
\draw [fill=qqqqff] (3.,3.) circle (3.5pt);
\draw [fill=qqqqff] (8.,6.) circle (3.5pt);
\draw [fill=qqqqff] (4.911764705882353,4.147058823529412) circle (3.5pt);
\end{scriptsize}
\end{tikzpicture}
\end{center}

\begin{proof}
Soit $y$ un élément de $\epsilon\text{-}\geod(x,a)\cap\epsilon\text{-}\geod(x',a)$. De l'inégalité triangulaire et de la définition de $\epsilon$-géodésique, on déduit
$$d(x,x') + 2 \cdot d(y,a) \leq \underbrace{d(x,y)+d(y,a)}_{\leq d(x,a)+\epsilon}+\underbrace{d(y,x')+d(y,a)}_{\leq d(x',a)+\epsilon},$$
\noindent
puis, en utilisant que $a \in \epsilon\text{-}\geod(x,x')$, il vient
$$d(x,x') + 2 \cdot d(y,a) \leq d(x,x') + 3 \cdot \epsilon.$$
\end{proof}

\vspace{0.2cm}

Nous introduisons à présent une classe d'espaces hyperboliques d'un type particulier (contenant les graphes de Cayley des groupes hyperboliques), ce qui nous permettra d'énoncer et de démontrer le théorème~\ref{enonce-bon} dans le paragraphe suivant ; de la proposition~\ref{bon-espace-hyp-discret}, nous déduirons finalement le théorème~\ref{enonce-ppal} du théorème~\ref{enonce-bon}.

\begin{defi}
Un bon espace hyperbolique discret est un espace hyperbolique, uniformément localement fini dont la métrique provient d'une structure de graphe.
\end{defi}

Deux remarques sur la définition précédente :

\begin{itemize}
\renewcommand{\labelitemi}{$\bullet$}
\item puisque la métrique $d : X \longrightarrow \N$ provient d'une structure de graphe, $(X,d)$ est {\it géodésique} dans le sens que, pour tous $x,y$ de $X$ et tout entier $k$ de $[\![0,\cdots,d(a,b)]\!]$, il existe un élément $z$ de $X$ tel que $d(x,z)=k$ et $d(z,y)=d(x,y)-k$ ;
\item comme $d$ est géodésique, l'uniforme locale finitude équivaut au fait que le nombre de points à distance $1$ d'un élément donné de $X$ est borné indépendamment de ce point.
\end{itemize}

\vspace{0.2cm}

\noindent
{\bf Exemple fondamental} : si $\Gamma$ est un groupe hyperbolique et $d$ la distance invariante à gauche associée à la longueur des mots pour un système fini de générateurs donné, alors $(\Gamma,d)$ est un bon espace hyperbolique discret ; en outre ce dernier est muni d'une action isométrique de $\Gamma$ par translation à gauche ($d$ provient de la structure de graphe de Cayley sur $\Gamma$ associée au système de générateurs donné).
On note enfin que cette action est propre, c'est-à-dire que la distance $d(o,g \cdot o)$ tend vers l'infini quand $g$ tend vers l'infini étant donné un sommet $o$ dans $(\Gamma,d)$.


\section{Des arbres aux bons espaces hyperboliques discrets}\label{demonstration-theoreme}

Dans ce paragraphe, nous énonçons et démontrons le théorème principal dans le cadre des bons espaces hyperboliques discrets. Les définitions concernant les espaces hyperboliques ont été rappelées au paragraphe~\ref{definitions-lemmes}.

\begin{thm}\label{enonce-bon}
Soit $(X,d)$ un bon espace hyperbolique discret muni d'une action isométrique d'un groupe $G$.
Notons $\delta$ un entier non nul tel que $(X,d)$ soit $\delta$-hyperbolique.
Alors il existe une application $G$-équivariante $(x,a)\mapsto \mu_x(a)$ de $X\times X$ dans les mesures de probabilité sur $X$ telle que
\begin{enumerate}
\renewcommand{\theenumi}{\roman{enumi}}
\item le support de $\mu_x(a)$ est inclus dans $B(a,4\delta)\cap2\delta\text{-}\geod(x,a)$ et même dans $S(a,4\delta)\cap2\delta\text{-}\geod(x,a)$ si $d(x,a)\geq 4\delta$ ;
\item il existe un réel $\epsilon$ strictement positif et une constante $C$ tels que pour tous $x,x',a$ de~$X$, si $d(x,x')=1$ alors $\|\mu_x(a)-\mu_{x'}(a)\|_{1}\leq C \text{e}^{-\epsilon d(x,a)}$ ;
\item il existe un réel $\eta$ strictement positif tel que, si la distance $d(x,x')$ est suffisamment grande, alors le nombre de $a$ de $X$ tels que les supports de $\mu_x(a)$ et $\mu_{x'}(a)$ sont disjoints est supérieur à $\eta \cdot d(x,x')$. 
\end{enumerate}
\end{thm}

\begin{cor}\label{cor-enonce-bon}
Soit $(X,d)$ un bon espace hyperbolique discret muni d'une action isométrique propre d'un groupe $G$ et $\delta$ un entier non nul tel que $(X,d)$ soit $\delta$-hyperbolique.
Alors $G$ admet une action affine propre sur un espace de Banach affine sur $\ell^p(X^{\leq 4\delta})$ pour tout~$p$ suffisamment grand.
\end{cor}

\begin{proof}
Notons $\mathcal{E}^{\circ}$ l'espace de Banach $\ell^p(X^{\leq 4\delta})$ ($p \in [1,\infty[$), $V$ l'espace vectoriel des fonctions sur $X^{\leq 4\delta}$ et $\xi_{\circ} : (x,y) \longmapsto \mu_o(x)(\{y\})$, où $o$ désigne un point de $X$ avec $\mu$ donnée par le théorème~\ref{enonce-ppal}.
L'item (i) du théorème~\ref{enonce-ppal} permet de voir $\xi_{\circ}$ comme un élément de~$V$ ; par ailleurs, pour tout $g$ de $G$, en utilisant la $G$-équivariance de $\mu$, la fonction $c(g) = \xi_{\circ} - g \cdot \xi_{\circ}$ se réécrit $c(g) = \mu_o - \mu_{g \cdot o}$.
L'item (ii) du théorème~\ref{enonce-ppal} assure que $c(g)$ est $\ell^p$-intégrable dès que $\sum_{a \in X}\text{e}^{-p \epsilon d(o,a)}$ est finie, ce qui sera toujours le cas pour $p$ suffisamment grand compte-tenu de la croissance au plus exponentielle des sphères dans $X$.
De plus, le cocycle $c : G \longrightarrow \mathcal{E}^{\circ}$ est localement constant et {\it a fortiori} continu, comme conséquence de la continuité de $g \mapsto g \cdot o$ de $G$ dans $X$ discret.
Enfin, pour tous $g$ de $G$ et $a$ de $X$, si les supports de $\mu_o(a)$ et $\mu_{g \cdot o}(a)$ sont disjoints, alors $\|\mu_o(a)-\mu_{g \cdot o}(a)\|_{1} = 2$ ; de l'item (iii) du théorème~\ref{enonce-ppal}, on déduit facilement que $\| c(g) \|_p$ est, à une constante multiplicative près, minorée par $d(o,g \cdot o)^{1/p}$, assurant ainsi que le 1-cocycle $c : G \longrightarrow \ell^p(X^{\leq 4\delta})$ est propre puisque l'action de $G$ sur $X$ est propre.
\end{proof}

\begin{cor}\cite{yu}
Tout groupe hyperbolique $\Gamma$ admet une action affine propre sur un espace $\ell^{p}$ pour $p$ suffisamment grand (et même plus précisément sur $\ell^{p}$ d'une réunion finie de copies de $\Gamma$).
\end{cor}

\begin{proof}
C'est une conséquence immédiate du corollaire précédent \ref{cor-enonce-bon} et de l'exemple fondamental de bon espace hyperbolique discret.
La preuve est complète après avoir noté que, pour tout entier naturel $R$, le $\Gamma$-espace $X^{\leq R}$ s'identifie à une réunion finie de copies de $\Gamma$.
En effet, l'application $(x,y) \longmapsto (x,x^{-1}y)$ de $X^{\leq R}  = \{(x,y)\in \Gamma \times \Gamma ; d(x,y)\leq R\}$ dans $\Gamma \times B(1,R)$ est une bijection $\Gamma$-équivariante, où l'action de $\Gamma$ sur $\Gamma \times B(1,R)$ est par translation à gauche sur le premier facteur.
\end{proof}

La suite de ce paragraphe est consacrée à la démonstration du théorème~\ref{enonce-bon}.

\begin{proof}
Nous introduisons quelques notations qui nous seront utiles par la suite.
Tout d'abord la fonction $\underline{\delta}$ définie sur les entiers naturels par $\underline{\delta}(n) = (4+5n)\delta$.
Si $t$ est un nombre réel positif, on note $n_{<t}$ le plus grand entier naturel $n$ tel que $\underline{\delta}(n)< t$.
Enfin si $A$ est une partie non vide de $X$ on note $\nu_A$ la mesure de probabilité uniforme sur $A$, égale au produit par $(\sharp A)^{-1}$ de la fonction caractéristique de $A$.

Étant donné un point $a$ de $X$, on définit une application $T_{a}$ de l'ensemble des mesures de probabilité sur $X$ dans lui-même en posant, pour tout point~$x$ de $X$,
\begin{itemize}
\renewcommand{\labelitemi}{$\bullet$}
\item $T_{a}(\delta_{x})=\delta_{x}$ si $d(x,a)\leq 4\delta$ ;
\item $T_{a}(\delta_{x})=\nu_{A}$, où $A=\delta\text{-}\geod(x,a) \cap S(a,\underline{\delta}(n_{<d(x,a)}))$ sinon,
\end{itemize}

\noindent
et en prolongeant $T_a$ par \og linéarité \fg\ à l'ensemble des mesures de probabilité.
Cette construction est par définition $G$-équivariante.


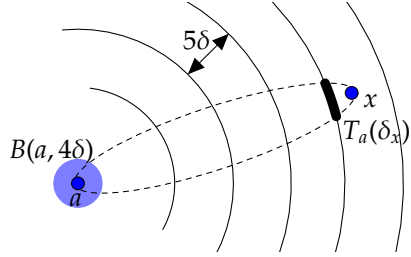
\begin{center}
\definecolor{xdxdff}{rgb}{0.49019607843137253,0.49019607843137253,1.}
\definecolor{qqqqff}{rgb}{0.,0.,1.}
\begin{tikzpicture}[line cap=round,line join=round,>=triangle 45,x=1.0cm,y=1.0cm,scale=0.6]
\clip(1.5,1.5) rectangle (11.,7.);
\draw [color=xdxdff,fill=xdxdff,fill opacity=1.0] (3.,3.) circle (0.5283937925449161cm);
\draw [rotate around={18.434948822922017:(6.,4.)},dash pattern=on 2pt off 2pt] (6.,4.) ellipse (3.236925187912023cm and 0.69114735920737cm);
\draw [shift={(3.,3.)}] plot[domain=-0.5016040541891202:1.4440287916935348,variable=\t]({1.*2.121320343559643*cos(\t r)+0.*2.121320343559643*sin(\t r)},{0.*2.121320343559643*cos(\t r)+1.*2.121320343559643*sin(\t r)});
\draw [shift={(3.,3.)}] plot[domain=-0.693986961537167:1.6710876442118139,variable=\t]({1.*3.4083427057735847*cos(\t r)+0.*3.4083427057735847*sin(\t r)},{0.*3.4083427057735847*cos(\t r)+1.*3.4083427057735847*sin(\t r)});
\draw [shift={(3.,3.)}] plot[domain=-0.6366577596779681:1.7514533619708839,variable=\t]({1.*4.676109494013158*cos(\t r)+0.*4.676109494013158*sin(\t r)},{0.*4.676109494013158*cos(\t r)+1.*4.676109494013158*sin(\t r)});
\draw [shift={(3.,3.)}] plot[domain=-0.7201157591188156:1.809776775654587,variable=\t]({1.*5.853170081246572*cos(\t r)+0.*5.853170081246572*sin(\t r)},{0.*5.853170081246572*cos(\t r)+1.*5.853170081246572*sin(\t r)});
\draw [shift={(3.,3.)},line width=3.6pt]  plot[domain=0.25537376795162037:0.3881273408416643,variable=\t]({1.*5.853170081246571*cos(\t r)+0.*5.853170081246571*sin(\t r)},{0.*5.853170081246571*cos(\t r)+1.*5.853170081246571*sin(\t r)});
\draw [shift={(3.,3.)}] plot[domain=-0.41850890555333464:1.9107266962645173,variable=\t]({1.*7.1358531375022*cos(\t r)+0.*7.1358531375022*sin(\t r)},{0.*7.1358531375022*cos(\t r)+1.*7.1358531375022*sin(\t r)});
\draw (9.06,5.18) node[anchor=north west] {$x$};
\draw (8.6,4.6) node[anchor=north west] {$T_a(\delta_x)$};
\draw [->] (5.85,5.85) -- (6.306508732787499,6.306508732787499);
\draw [->] (5.85,5.85) -- (5.410062239860208,5.410062239860208);
\draw (5.1,6.6) node[anchor=north west] {$5\delta$};
\draw (1.3,4.2) node[anchor=north west] {$B(a,4\delta)$};
\draw (2.62,3) node[anchor=north west] {$a$};
\begin{scriptsize}
\draw [fill=qqqqff] (3.,3.) circle (4pt);
\draw [fill=qqqqff] (9.,5.) circle (4pt);
\end{scriptsize}
\end{tikzpicture}
\end{center}

On définit enfin
$$\mu_x(a) = \lim_{l \to \infty} T_{a}^l (\delta_{x}), \quad \text{où} \quad T_{a}^l = \underbrace{T_{a} \circ \cdots \circ T_{a}}_{l \, \text{fois}}.$$
L'existence de cette limite est évidente puisque par construction la suite est stationnaire pour $l > n_{<d(x,a)}$.
Si $d(x,a)\geq 4\delta$ et $1 \leq l \leq n_{<d(x,a)}+1$, on remarque que $T_{a}^{l}(\delta_{x})$ est supportée sur la sphère $S(a,\underline{\delta}(n_{<d(x,a)}+1-l))$ ; la mesure $\mu_x(a)$ est donc dans ce cas une mesure de probabilité supportée sur $S(a,\underline{\delta}(0))$.


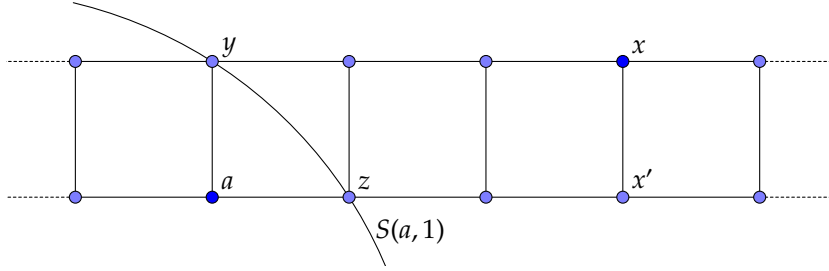
\begin{center}
\definecolor{xdxdff}{rgb}{0.49019607843137253,0.49019607843137253,1.}
\definecolor{qqqqff}{rgb}{0.,0.,1.}
\begin{tikzpicture}[line cap=round,line join=round,>=triangle 45,x=1.0cm,y=1.0cm,scale=0.9]
\clip(2.,2.) rectangle (14.,6.);
\draw (3.,5.)-- (5.,5.);
\draw (7.,5.)-- (5.,5.);
\draw (7.,5.)-- (9.,5.);
\draw (9.,5.)-- (11.,5.);
\draw (11.,5.)-- (13.,5.);
\draw [dash pattern=on 1pt off 1pt] (3.,5.)-- (1.,5.);
\draw [dash pattern=on 1pt off 1pt] (1.,3.)-- (3.,3.);
\draw (3.,3.)-- (5.,3.);
\draw (5.,3.)-- (7.,3.);
\draw (7.,3.)-- (9.,3.);
\draw (9.,3.)-- (11.,3.);
\draw (11.,3.)-- (13.,3.);
\draw [dash pattern=on 1pt off 1pt] (13.,3.)-- (15.,3.);
\draw [dash pattern=on 1pt off 1pt] (13.,5.)-- (15.,5.);
\draw (3.,5.)-- (3.,3.);
\draw (5.,5.)-- (5.,3.);
\draw (7.,5.)-- (7.,3.);
\draw (9.,5.)-- (9.,3.);
\draw (11.,5.)-- (11.,3.);
\draw (13.,5.)-- (13.,3.);
\draw [shift={(1.42,-0.58)}] plot[domain=0.3064845648498623:1.3351363903325004,variable=\t]({1.*6.629690792186314*cos(\t r)+0.*6.629690792186314*sin(\t r)},{0.*6.629690792186314*cos(\t r)+1.*6.629690792186314*sin(\t r)});
\draw (7.25,2.823067609584557) node[anchor=north west] {$S(a,1)$};
\draw (5,3.45) node[anchor=north west] {$a$};
\draw (11,3.55) node[anchor=north west] {$x'$};
\draw (11,5.45) node[anchor=north west] {$x$};
\draw (7,3.45) node[anchor=north west] {$z$};
\draw (5,5.5) node[anchor=north west] {$y$};
\begin{scriptsize}
\draw [fill=xdxdff] (3.,5.) circle (2.5pt);
\draw [fill=xdxdff] (3.,3.) circle (2.5pt);
\draw [fill=xdxdff] (5.,5.) circle (2.5pt);
\draw [fill=xdxdff] (7.,5.) circle (2.5pt);
\draw [fill=xdxdff] (9.,5.) circle (2.5pt);
\draw [fill=qqqqff] (11.,5.) circle (2.5pt);
\draw [fill=xdxdff] (13.,5.) circle (2.5pt);
\draw [fill=qqqqff] (5.,3.) circle (2.5pt);
\draw [fill=xdxdff] (7.,3.) circle (2.5pt);
\draw [fill=xdxdff] (9.,3.) circle (2.5pt);
\draw [fill=xdxdff] (11.,3.) circle (2.5pt);
\draw [fill=xdxdff] (13.,3.) circle (2.5pt);
\end{scriptsize}
\end{tikzpicture}
\end{center}

\noindent
{\bf Remarque} : si on considère l'espace métrique $\Z \times \Z/{2\Z}$, on comprend qu'itérer une seule fois l'application $T_{a}$ sur $\delta_x$ n'est pas suffisant pour pouvoir construire une mesure satisfaisant la deuxième condition de l'énoncé puisqu'avec les notations du dessin, on a $T_{a}(\delta_x) = \frac{1}{2}(\delta_y + \delta_z)$ et $T_{a}(\delta_{x'}) = \delta_z$.
C'est tout l'intérêt de la construction de $\mu_x(a)$ ci-dessus qui, en découpant l'espace en couronnes d'épaisseur~$5\delta$, \og homogénéise \fg\ la mesure à chaque itération et permet aux mesures $\mu_x(a)$ et $\mu_{x'}(a)$ d'être exponentiellement proches au fur et à mesure que le point $a$ s'éloigne des points $x$ et $x'$ voisins l'un de l'autre.

\vspace{0.2cm}

Nous commençons par démontrer que $\mu$ satisfait la première assertion du théorème~\ref{enonce-bon} : c'est une conséquence immédiate du lemme suivant.

\begin{lem}\label{iteration1}
Pour tout entier naturel $l$, le support de $T_{a}^{l}(\delta_{x})$ est contenu dans $2\delta\text{-}\geod(x,a)$.
\end{lem}

C'est évident pour $l=0$ et $l=1$ puisque $\delta\text{-}\geod(x,a) \subset 2\delta\text{-}\geod(x,a)$.
Le lemme~\ref{chemin-alpha-geodesique} (avec $\alpha=2\delta$, $\beta=\delta$) permet de conclure par récurrence sur $l$.
En effet, si $z$ est dans le support de $T_{a}^{l+1}(\delta_{x})$, désignons par $y$ un élément du support de $T_{a}^{l}(\delta_{x})$ tel que $z$ soit dans le support de $T_{a}(\delta_{y})$.
Mais alors $y$ appartient à $\alpha\text{-}\geod(a,x)$, $z$ à $\beta\text{-}\geod(a,y)$ et $d(y,z)\geq 5\delta\geq 3\delta/2$.
\cqfd

\vspace{0.2cm}

Le troisième point du théorème~\ref{enonce-bon} se déduit du premier point et du lemme~\ref{milieu-geodesique} : en effet, les supports de $\mu_x(a)$ et $\mu_{x'}(a)$ sont alors disjoints pour tous les points $a$ appartenant à $2\delta\text{-}\geod(x,x')$ et qui sont à distance supérieure à $4\delta$ de $x$ et $x'$.
\vspace{0.2cm}

Il nous reste donc à voir que la deuxième assertion de l'énoncé~\ref{enonce-bon} est également satisfaite.
Pour cela, on appelle {\it distance maximale} entre deux parties finies~$A$ et~$B$ de~$X$ le nombre $\max_{a\in A,b\in B}d(a,b)$.
On commence par démontrer que pour deux points~$x$, $x'$ proches, on peut contrôler la distance maximale entre les supports de $T_{a}(\delta_{x})$ et de $T_{a}(\delta_{x'})$.
Plus précisément, on a le lemme suivant.

\begin{lem}\label{lem3}
Soit $x,x'$ deux points voisins de $X$ ($d(x,x')=1$) et $n$ un entier naturel.
\begin{enumerate}
\renewcommand{\theenumi}{\roman{enumi}}
\item Si $x$ et $x'$ appartiennent à la couronne $B(a,\underline{\delta}(n+1)) \setminus B(a,\underline{\delta}(n))$, alors les mesures $T_{a}(\delta_{x})$ et $T_{a}(\delta_{x'})$ sont supportées sur la sphère $S(a,\underline{\delta}(n))$ et la distance maximale entre les supports de ces mesures est inférieure à $4\delta$.
\item Si $x$ appartient à la sphère $S(a,\underline{\delta}(n))$ et $x'$ à la sphère $S(a,\underline{\delta}(n)+1)$, même conclusion pour les mesures $\delta_{x}$ et $T_{a}(\delta_{x'})$ qu'au point (i).
\end{enumerate}
\noindent
\end{lem}


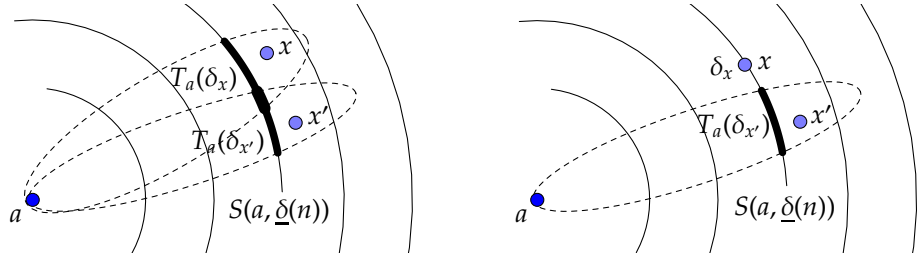
\definecolor{xdxdff}{rgb}{0.49019607843137253,0.49019607843137253,1.}
\definecolor{qqqqff}{rgb}{0.,0.,1.}
\begin{tikzpicture}[line cap=round,line join=round,>=triangle 45,x=1.0cm,y=1.0cm,scale=0.7]
\clip(2.,2.) rectangle (20.,6.7);
\draw [rotate around={18.434948822922017:(6.,4.)},dash pattern=on 2pt off 2pt] (6.,4.) ellipse (3.236925187912023cm and 0.69114735920737cm);
\draw [shift={(3.,3.)}] plot[domain=-0.5016040541891202:1.4440287916935348,variable=\t]({1.*2.121320343559643*cos(\t r)+0.*2.121320343559643*sin(\t r)},{0.*2.121320343559643*cos(\t r)+1.*2.121320343559643*sin(\t r)});
\draw [shift={(3.,3.)}] plot[domain=-0.693986961537167:1.6710876442118139,variable=\t]({1.*3.4083427057735847*cos(\t r)+0.*3.4083427057735847*sin(\t r)},{0.*3.4083427057735847*cos(\t r)+1.*3.4083427057735847*sin(\t r)});
\draw [shift={(3.,3.)}] plot[domain=0.032292564510163956:1.7514533619708839,variable=\t]({1.*4.691168648876256*cos(\t r)+0.*4.691168648876256*sin(\t r)},{0.*4.691168648876256*cos(\t r)+1.*4.691168648876256*sin(\t r)});
\draw [shift={(3.,3.)}] plot[domain=-0.7201157591188156:1.809776775654587,variable=\t]({1.*5.853170081246572*cos(\t r)+0.*5.853170081246572*sin(\t r)},{0.*5.853170081246572*cos(\t r)+1.*5.853170081246572*sin(\t r)});
\draw [shift={(3.,3.)}] plot[domain=-0.41850890555333464:1.9107266962645173,variable=\t]({1.*7.1358531375022*cos(\t r)+0.*7.1358531375022*sin(\t r)},{0.*7.1358531375022*cos(\t r)+1.*7.1358531375022*sin(\t r)});
\draw (8,4.964860102005301) node[anchor=north west] {$x'$};
\draw (5.8,4.5) node[anchor=north west] {$T_a(\delta_{x'})$};
\draw (2.4,3) node[anchor=north west] {$a$};
\draw [rotate around={30.76271953423893:(5.514793701027475,4.496901012516354)},dash pattern=on 2pt off 2pt] (5.514793701027475,4.496901012516354) ellipse (3.0597393526665777cm and 0.8928073175419783cm);
\draw [shift={(3.,3.)},line width=2.8pt]  plot[domain=0.19054808344889865:0.4529530253443871,variable=\t]({1.*4.691168648876257*cos(\t r)+0.*4.691168648876257*sin(\t r)},{0.*4.691168648876257*cos(\t r)+1.*4.691168648876257*sin(\t r)});
\draw [shift={(3.,3.)},line width=2.8pt]  plot[domain=0.3782431708270035:0.6955783146530887,variable=\t]({1.*4.691168648876254*cos(\t r)+0.*4.691168648876254*sin(\t r)},{0.*4.691168648876254*cos(\t r)+1.*4.691168648876254*sin(\t r)});
\draw (7.45,6.182603068669704) node[anchor=north west] {$x$};
\draw [shift={(3.,3.)},line width=4.4pt]  plot[domain=0.3782431708270035:0.4529530253443871,variable=\t]({1.*4.691168648876254*cos(\t r)+0.*4.691168648876254*sin(\t r)},{0.*4.691168648876254*cos(\t r)+1.*4.691168648876254*sin(\t r)});
\draw (5.4,5.7) node[anchor=north west] {$T_a(\delta_{x})$};
\draw (6.5,3.2) node[anchor=north west] {$S(a,\underline{\delta}(n))$};
\draw [rotate around={18.606663931722988:(15.48,4.01)},dash pattern=on 2pt off 2pt] (15.48,4.01) ellipse (3.2415086438012986cm and 0.6980532127556973cm);
\draw [shift={(12.48,3.)}] plot[domain=-0.4932999321890792:1.445246589686371,variable=\t]({1.*2.11177650332605*cos(\t r)+0.*2.11177650332605*sin(\t r)},{0.*2.11177650332605*cos(\t r)+1.*2.11177650332605*sin(\t r)});
\draw [shift={(12.48,3.)}] plot[domain=-0.6894592905441614:1.6704649792860589,variable=\t]({1.*3.3955853692699285*cos(\t r)+0.*3.3955853692699285*sin(\t r)},{0.*3.3955853692699285*cos(\t r)+1.*3.3955853692699285*sin(\t r)});
\draw [shift={(12.48,3.)}] plot[domain=0.057128773367495854:1.7506498265873751,variable=\t]({1.*4.694965972711289*cos(\t r)+0.*4.694965972711289*sin(\t r)},{0.*4.694965972711289*cos(\t r)+1.*4.694965972711289*sin(\t r)});
\draw [shift={(12.48,3.)}] plot[domain=-0.7175413405411444:1.8089433284853662,variable=\t]({1.*5.84*cos(\t r)+0.*5.84*sin(\t r)},{0.*5.84*cos(\t r)+1.*5.84*sin(\t r)});
\draw [shift={(12.48,3.)}] plot[domain=-0.41594513124665333:1.9097709256768987,variable=\t]({1.*7.127748592648313*cos(\t r)+0.*7.127748592648313*sin(\t r)},{0.*7.127748592648313*cos(\t r)+1.*7.127748592648313*sin(\t r)});
\draw (17.45,4.984823101458815) node[anchor=north west] {$x'$};
\draw (15.3,4.805156106377182) node[anchor=north west] {$T_a(\delta_{x'})$};
\draw (11.9,3) node[anchor=north west] {$a$};
\draw [shift={(12.48,3.)},line width=2.8pt]  plot[domain=0.19227916032066736:0.4572159365206113,variable=\t]({1.*4.694965972711289*cos(\t r)+0.*4.694965972711289*sin(\t r)},{0.*4.694965972711289*cos(\t r)+1.*4.694965972711289*sin(\t r)});
\draw (16.45,5.9) node[anchor=north west] {$x$};
\draw (15.55,5.9) node[anchor=north west] {$\delta_{x}$};
\draw (16,3.25) node[anchor=north west] {$S(a,\underline{\delta}(n))$};
\begin{scriptsize}
\draw [fill=qqqqff] (3.,3.) circle (3.5pt);
\draw [fill=xdxdff] (7.4,5.78) circle (3.5pt);
\draw [fill=xdxdff] (7.94,4.46) circle (3.5pt);
\draw [fill=qqqqff] (12.48,3.) circle (3.5pt);
\draw [fill=xdxdff] (16.38,5.56) circle (3.5pt);
\draw [fill=xdxdff] (17.42,4.48) circle (3.5pt);
\end{scriptsize}
\end{tikzpicture}

Le seul point qui n'est pas complètement évident est de montrer, dans le premier cas, que la distance maximale entre les supports de $T_{a}(\delta_{x})$ et $T_{a}(\delta_{x'})$  est inférieure à $4\delta$.
Remarquons que $\delta\text{-}\geod(x',a) \subset (\delta+2)\text{-}\geod(x,a)$ puisqu'en notant $y'$ un élément du premier ensemble, on a
$$d(a,y')+d(y',x) \leq \underbrace{d(a,y') + d(y',x')}_{\leq d(a,x') + \delta} + d(x',x) \leq d(a,x) + \delta + 2.$$
On applique alors le lemme \ref{sphere-alpha-geodesique} (avec $\alpha=\delta+2, y\in \mathrm{supp}(T_{a}(\delta_{x}))$, $y'\in \mathrm{supp}(T_{a}(\delta_{x'}))$ en notant que $d(a,y)=d(a,y')=\underline{\delta}(n)$) et on conclut en remarquant que $2\delta+2\leq 4\delta$ puisque $\delta$ est un entier non nul.
\cqfd

\vspace{0.2cm}

Soit $K$ un entier naturel tel que tout point de $X$ a au plus $K$ voisins (autrement dit les sphères de rayon~1 sont de cardinal au plus $K$).
Nous démontrons maintenant le lemme-clé qui nous permettra de conclure.

\begin{lem}\label{lem-n}
Soit $n$ un entier naturel non nul et $x,x'$ deux points de $S(a,\underline{\delta}(n))$.
Si $d(x,x')\leq 4\delta$, alors les supports des mesures $T_{a}(\delta_{x})$ et $T_{a}(\delta_{x'})$ sont contenus dans la sphère $S(a,\underline{\delta}(n-1))$ et ont des cardinaux bornés par une constante $C(\delta,K)$ ne dépendant que de $\delta$ et $K$.
De plus, ces supports ont une distance maximale entre eux inférieure à $4\delta$, ainsi qu'une intersection commune non triviale.
\end{lem}


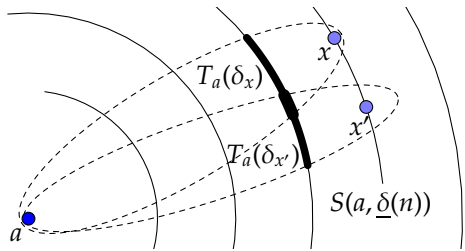
\begin{center}
\definecolor{xdxdff}{rgb}{0.49019607843137253,0.49019607843137253,1.}
\definecolor{qqqqff}{rgb}{0.,0.,1.}
\begin{tikzpicture}[line cap=round,line join=round,>=triangle 45,x=1.0cm,y=1.0cm,scale=0.8]
\clip(2.5,2.5) rectangle (10.5,6.5);
\draw [rotate around={18.434948822922017:(6.,4.)},dash pattern=on 2pt off 2pt] (6.,4.) ellipse (3.236925187912023cm and 0.69114735920737cm);
\draw [shift={(3.,3.)}] plot[domain=-0.5016040541891202:1.4440287916935348,variable=\t]({1.*2.121320343559643*cos(\t r)+0.*2.121320343559643*sin(\t r)},{0.*2.121320343559643*cos(\t r)+1.*2.121320343559643*sin(\t r)});
\draw [shift={(3.,3.)}] plot[domain=-0.693986961537167:1.6710876442118139,variable=\t]({1.*3.4083427057735847*cos(\t r)+0.*3.4083427057735847*sin(\t r)},{0.*3.4083427057735847*cos(\t r)+1.*3.4083427057735847*sin(\t r)});
\draw [shift={(3.,3.)}] plot[domain=-0.6366577596779681:1.7514533619708839,variable=\t]({1.*4.676109494013158*cos(\t r)+0.*4.676109494013158*sin(\t r)},{0.*4.676109494013158*cos(\t r)+1.*4.676109494013158*sin(\t r)});
\draw [shift={(3.,3.)}] plot[domain=0.1:1.809776775654587,variable=\t]({1.*5.853170081246572*cos(\t r)+0.*5.853170081246572*sin(\t r)},{0.*5.853170081246572*cos(\t r)+1.*5.853170081246572*sin(\t r)});
\draw [shift={(3.,3.)}] plot[domain=-0.41850890555333464:1.9107266962645173,variable=\t]({1.*7.1358531375022*cos(\t r)+0.*7.1358531375022*sin(\t r)},{0.*7.1358531375022*cos(\t r)+1.*7.1358531375022*sin(\t r)});
\draw (8.12,4.85) node[anchor=north west] {$x'$};
\draw (6.1,4.4) node[anchor=north west] {$T_a(\delta_{x'})$};
\draw (2.55,3) node[anchor=north west] {$a$};
\draw [rotate around={30.76271953423893:(5.514793701027475,4.496901012516354)},dash pattern=on 2pt off 2pt] (5.514793701027475,4.496901012516354) ellipse (3.0597393526665777cm and 0.8928073175419783cm);
\draw (7.8,3.6) node[anchor=north west] {$S(a,\underline{\delta}(n))$};
\draw [shift={(3.,3.)},line width=2.8pt]  plot[domain=0.18975780083566285:0.45374330795762396,variable=\t]({1.*4.676109494013157*cos(\t r)+0.*4.676109494013157*sin(\t r)},{0.*4.676109494013157*cos(\t r)+1.*4.676109494013157*sin(\t r)});
\draw [shift={(3.,3.)},line width=2.8pt]  plot[domain=0.3770641327288447:0.6967573527512481,variable=\t]({1.*4.676109494013162*cos(\t r)+0.*4.676109494013162*sin(\t r)},{0.*4.676109494013162*cos(\t r)+1.*4.676109494013162*sin(\t r)});
\draw (7.6,6) node[anchor=north west] {$x$};
\draw [shift={(3.,3.)},line width=4.4pt]  plot[domain=0.3770641327288447:0.45374330795762396,variable=\t]({1.*4.676109494013162*cos(\t r)+0.*4.676109494013162*sin(\t r)},{0.*4.676109494013162*cos(\t r)+1.*4.676109494013162*sin(\t r)});
\draw (5.6,5.7) node[anchor=north west] {$T_a(\delta_{x})$};
\begin{scriptsize}
\draw [fill=qqqqff] (3.,3.) circle (3pt);
\draw [fill=xdxdff] (8.552804696727591,4.850934898909197) circle (3pt);
\draw [fill=xdxdff] (8.029587402054963,5.993802025032718) circle (3pt);
\end{scriptsize}
\end{tikzpicture}
\end{center}

Montrons qu'il existe une constante $C(\delta,K)$ qui majore le cardinal du support de $T_{a}(\delta_{x})$.
Or, pour tous $a,x$ de $X$, le support de $T_{a}(\delta_{x})$ a un diamètre inférieur ou égal à $2\delta$ : c'est évident si $d(x,a)\leq 4\delta$ et sinon, puisque le support de $T_{a}(\delta_{x})$ est inclus dans $\delta\text{-}\geod(x,a)\cap S(a,\underline{\delta}(n))$, une application du lemme~\ref{sphere-alpha-geodesique} avec $\alpha=\delta$ et $y,y'$ dans le support de $T_{a}(\delta_{x})$ permet de conclure. 

Soit $y$ et $y'$ respectivement dans le support de $T_{a}(\delta_{x})$ et $T_{a}(\delta_{x'})$.
Puisque $d(x,x')\leq 4\delta$, on en déduit que $x'$ appartient à $8\delta\text{-}\geod(x,a)$.
D'autre part, $y'$ appartient à $\delta\text{-}\geod(x',a)$ et $d(x',y')\geq 5\delta \geq (8\delta+\delta)/2$.
En appliquant le lemme~\ref{chemin-alpha-geodesique} à $a,x,x',y'$ (au lieu de $a,x,y,z$) avec $\alpha=8\delta$ et $\beta=\delta$, on en déduit que $y'$ appartient à $2\delta\text{-}\geod(x,a)$.
Le lemme~\ref{sphere-alpha-geodesique} appliqué à $a,x,y,y'$ avec $\alpha=2\delta$ montre alors $d(y,y')\leq 3\delta$.
La distance maximale entre les supports de $T_{a}(\delta_{x})$ et $T_{a}(\delta_{x'})$ est donc inférieure à $3\delta$ (et {\it a fortiori} $\leq 4\delta$).

Enfin montrons que l'intersection des supports de $T_{a}(\delta_{x})$ et $T_{a}(\delta_{x'})$ est non vide.
Pour cela, il suffit de voir que si $y'$ appartient à $\geod(x',a) \cap S(a,\underline{\delta}(n-1))$ (et donc au support de $T_{a}(\delta_{x'})$), alors $y'$ appartient à $\delta\text{-}\geod(x,a)$ (et donc au support de $T_{a}(\delta_{x})$).
C'est exactement le même argument que précédemment, c'est-à-dire une application du lemme~\ref{chemin-alpha-geodesique} avec $a,x,x',y $ (au lieu de $a,x,y,z$), $\alpha=8\delta$ et $\beta=0$.
\cqfd

\begin{lem}\label{dernier-lemme}
Il existe un réel $\epsilon$ strictement positif tel que pour tout entier naturel~$n$ et pour tous points $a$ de $X$ et $x,x'$ de $S(a,\underline{\delta}(n))$, si $d(x,x')\leq 4\delta$, alors $\|\mu_x(a)-\mu_{x'}(a)\|_{1}\leq 2\text{e}^{-\epsilon n}$.
\end{lem}

Remarquons que pour deux mesures de probabilité $\mu$ et $\nu$ à support fini sur $X$, $\frac{1}{2} \|\mu-\nu\|_{1}$ est la masse de la partie qui n'est pas commune aux deux mesures.
Par ailleurs, en écrivant $\mu-\nu=\sum_{i}\kappa_{i}(\delta_{x_{i}}-\delta_{y_{i}})$ avec $\kappa_{i}$ des réels strictement positifs, $x_i, y_i$ des points respectivement dans les supports de $\mu$ et $\nu$ tels que l'ensemble des $x_i$ soit disjoint de l'ensemble des $y_i$, on a alors $\sum_{i} \kappa_{i}=\frac{1}{2} \|\mu-\nu\|_{1}$.

Pour tout entier naturel $n$, si $x, x'$ sont des éléments de $S(a,\underline{\delta}(n))$ tels que $d(x,x')\leq 4\delta$, d'après le lemme~\ref{lem-n}, on a $\frac{1}{2} \|T_{a}(\delta_{x})-T_{a}(\delta_{x'})\|_{1}\leq (1-1/C)$, pour une certaine constante $C = C(\delta,K)$ qui, de plus, majore le cardinal du support de $T_{a}(\delta_{x})$.
Cette inégalité s'étend au cas de deux mesures $\mu$ et $\nu$ dont les cardinaux des supports sont majorés par $C$ et dont la distance entre les supports est inférieure à $4\delta$.
En effet, avec les notations précédentes, on a
$$\|T_{a}(\mu)-T_{a}(\nu)\|_{1} \leq \sum_{i}\kappa_{i}\,\|T_{a}(\delta_{x_{i}})-T_{a}(\delta_{y_{i}})\|_{1} \leq (1-1/C)\cdot \|\mu-\nu\|_{1}.$$

D'où la décroissance exponentielle annoncée avec $\epsilon = -\ln(1-1/C) >0$.
\cqfd

Le deuxième point du théorème~\ref{enonce-bon} est une conséquence des lemmes~\ref{lem3} et \ref{dernier-lemme}.
Ceci termine la preuve du théorème~\ref{enonce-bon}.
\end{proof}


\section{Généralisation du théorème principal}

Pour terminer, nous donnons une version plus générale (théorème~\ref{enonce-ppal}) du théorème principal~\ref{enonce-bon}.
Ce sera essentiellement une conséquence de la proposition suivante.

\begin{prop}\label{bon-espace-hyp-discret}
Soit $\delta$ un réel positif et $(X,d)$ un espace métrique $\delta$-hyperbolique, faiblement $\delta$-géodésique, uniformément localement fini, et muni d'une action isométrique d'un groupe $G$.
La fonction 
$$d'(x,y)=\min\{k\in \N ; \exists\  x_0,\cdots,x_k,$$
$$\text{tels que \ }x_0=x, \ x_k=y \text{\ et \ }\forall
j\in [\![0,\cdots,k-1]\!], \ d(x_j,x_{j+1})\leq \delta+1\}.$$
\noindent
est une distance quasi-isométrique à $d$.
De plus, l'espace métrique $(X,d')$ est un bon espace hyperbolique discret muni d'une action isométrique du groupe $G$.
\end{prop}

\begin{proof}
Par construction, $d'$ provient de la structure de graphe sur $X$ pour laquelle deux points distincts $x,y$ de $X$ sont voisins si $d(x,y)\leq \delta +1$.
L'inégalité triangulaire assure que $d\leq (\delta+1)\cdot d'$.
Puisque $(X,d)$ est faiblement $\delta$-géodésique, si $x,y$ sont deux points de $(X,d)$ tels que $d(x,y)>1$, alors il existe un point $z$ de~$X$ tel que $d(z,x) \leq \delta+1$ et $d(z,y) \leq d(x,y)+\delta-(\delta+1)=d(x,y)-1$.
On en déduit donc $d'\leq d+1$ : les métriques $d$ et $d'$ sont quasi-isométriques.
L'hyperbolicité de $(X,d')$ résulte de la conservation de l'hyperbolicité par quasi-isométrie pour les espaces faiblement géodésiques.
La démonstration de ce dernier point figure dans~\cite{harpehyper,delzanthyper} pour les espaces géodésiques ; la preuve est essentiellement la même dans le cas des espaces faiblement géodésiques.
On peut aussi invoquer le théorème 3.18 de~\cite{vaisala} ainsi que  la remarque 3.19 de~\cite{vaisala} appliquée à $(X,d)$ et à l'espace total du graphe considéré précédemment.
\end{proof}

\begin{thm}\label{enonce-ppal}
\renewcommand{\theenumi}{\roman{enumi}}
Soit $(X,d)$ un espace hyperbolique faiblement géodésique et uniformément localement fini muni d'une action isométrique d'un groupe $G$.
Alors il existe une application $G$-équivariante $(x,a)\mapsto \mu_x(a)$ de $X\times X$ dans l'espace des mesures de probabilité sur $X$ telle que :
\begin{enumerate}
\item il existe un entier naturel $R$ tel que pour tous $a$ et $x$ dans $X$, le support de $\mu_x(a)$ est contenu dans la boule fermée centrée en $a$ et de rayon $R$ ;
\item il existe $\epsilon >0$ tel que, pour tout entier naturel $k$, il existe $C>0$ tel que pour tous $x,x'$ et $a$ de $X$, si $d(x,x')\leq k$ alors $\|\mu_x(a)-\mu_{x'}(a)\|_{1}\leq C \cdot \text{e}^{-\epsilon d(x,a)}$ ;
\item il existe $\eta>0$ tel que, si la distance $d(x,x')$ est suffisamment grande, alors le nombre de $a$ de $X$ tels que les supports de $\mu_x(a)$ et $\mu_{x'}(a)$ sont disjoints est supérieur à $\eta \cdot d(x,x')$. 
\end{enumerate}
\end{thm}

\begin{proof}
Le théorème~\ref{enonce-ppal} pour $(X,d)$ résulte du théorème~\ref{enonce-bon} ci-dessus appliqué à $(X,d')$ grâce à la proposition~\ref{bon-espace-hyp-discret}.
\end{proof}

On en déduit le corollaire suivant qui se démontre {\it mutatis mutandis} comme son analogue le corollaire~\ref{cor-enonce-bon}.

\begin{cor}\cite{yu}
Soit $(X,d)$ un espace hyperbolique faiblement géodésique et uniformément localement fini muni d'une action isométrique propre d'un groupe $G$.
Alors $G$ admet une action affine propre sur un espace $\ell^p(X^{\leq R})$ pour tout $p$ suffisamment grand et un certain entier naturel $R$ donné.
\end{cor}


\end{document}